\documentclass[a4paper,leqno]{amsart}
\usepackage{amsmath, amssymb, amsthm}

\usepackage[mathscr]{eucal}

\newtheorem{thm}{\textbf{Theorem}}[section]
\newtheorem{prop}[thm]{\textbf{Proposition}}
\newtheorem{lem}[thm]{\textbf{Lemma}}
\newtheorem{cor}[thm]{\textbf{Corollary}}

\theoremstyle{definition}
\newtheorem{defn}[thm]{{\rm Definition}}
\newtheorem{ex}[thm]{{\rm Example}}

\newcommand{\olapla}{\overline\bigtriangleup}
\newcommand{\onabla}{\overline\nabla}
\newcommand{\wnabla}{\widetilde\nabla}
\newcommand{\lapla}{\bigtriangleup}

\newcommand{\p}{\phi}

\newcommand{\vp}{\varphi}

\title{k-harmonic immersion and submersion into a sphere}

\author{Shun Maeta}

\curraddr{Nakakuki 3-10-9 Oyama-shi Tochigi
Japan}
\email{shun.maeta@gmail.com}


\begin{document}




\begin{abstract}
 J. Eells and L. Lemaire introduced k-harmonic maps,
 and Wang Shaobo showed the first variational formula.
 When, k=2, it is called biharmonic maps (2-harmonic maps). There have been
 extensive studies in the area.
 In this paper, we study k-harmonic immersion into a sphere, and get the relationship between 
 radius and "k" of k-harmonic. And we also consider k-harmonic submersion.
Furthermore, we construct non harmonic k-harmonic by Hopf map. 
\end{abstract}

\maketitle


\vspace{10pt}
\begin{flushleft}
{\large {\bf Introduction}}
\end{flushleft}
Theory of harmonic maps has been applied into various fields in differential geometry.
 The harmonic maps between two Riemannian manifolds are
 critical maps of the energy functional $E(\p)=\frac{1}{2}\int_M\|d\p\|^2v_g$, for smooth maps $\p:M\rightarrow N$.
 
On the other hand, in 1981, J. Eells and L. Lemaire \cite{jell1} proposed the problem to consider the {\em $k$-harmonic maps}:
 they are critical maps of the functional 
 \begin{align*}
 E_{k}(\p)=\int_Me_k(\p)v_g,\ \ (k=1,2,\dotsm),
 \end{align*}
 where $e_k(\p)=\frac{1}{2}\|(d+d^*)^k\p\|^2$ for smooth maps $\p:M\rightarrow N$.
G.Y. Jiang \cite{jg1} studied the first and second variational formulas of the bi-energy $E_2$, 
and critical maps of $E_2$ are called {\em biharmonic maps} ({\em 2-harmonic maps}). There have been extensive studies on biharmonic maps.
 
In 1989, Wang Shaobo \cite{ws1} studied the first variational formula of the 
$k$-energy $E_k$,
 whose critical maps are called $k$-harmonic maps.
 Harmonic maps are always $k$-harmonic maps by definition.
 Especially, harmonic maps are always biharmonic maps.
 Currently, there are a lot of paper of biharmonic maps.
 But in this paper, we show biharmonic is not always $k$-harmonic $(k\geq 3)$.
 More generally, $s$-harmonic is not always $k$-harmonic $(s<k)$.
 This may imply that we must study $k$-harmonic maps  for generalized theory of harmonic maps.
 
In this paper, we study $k$-harmonic immersion and submersion.

In $\S \ref{preliminaries}$, we introduce notation and fundamental formulas of the tension field.

In $\S \ref{k-harmonic}$, we recall $k$-harmonic maps.

In $\S \ref{immersion}$, we study $k$-harmonic isometric immersion into a sphere, and obtain the necessary and sufficient condition
 for $\p$ to be $k$-harmonic map.
Furthermore we get several examples of $k$-harmonic maps.

Finally, in $\S \ref{submersion}$, we study $k$-harmonic submersion.
 And we construct a non harmonic $k$-harmonic map by Hopf map.

\vspace{30pt}
\section{Preliminaries}\label{preliminaries}
Let $(M,g)$ be an $m$ dimensional Riemannian manifold,
 $(N,h)$ an $n$ dimensional one,
 and $\p:M\rightarrow N$, a smooth map.
 We use the following notation.
 The second fundamental form $B(\p)$
 of $\p$ is a covariant differentiation $\widetilde\nabla d\p$ of $1$-form $d\p$,
 which is a section of $\odot ^2T^*M\otimes \p^{-1}TN$. And we denote $B(\p)$ by $B$.
For every $X,Y\in \Gamma (TM)$, let
 \begin{equation}
 \begin{split}
 B(X,Y)
&=(\widetilde\nabla d\p)(X,Y)=(\widetilde\nabla_X d\p)(Y)\\
&=\overline\nabla_Xd\p(Y)-d\p(\nabla_X Y)=\nabla^N_{d\p(X)}d\p(Y)-d\p(\nabla_XY). 
 \end{split}
 \end{equation}
 Here, $\nabla, \nabla^N, \overline \nabla, \widetilde \nabla$ are the induced connections on the bundles $TM$,
 $TN$, $\p^{-1}TN$ and $T^*M\otimes \p^{-1}TN$, respectively.
 
 And we also denote $B(e_i,e_j)$ by $B_{ij}$, where $\{e_i\}_{i=1}^m$ is locally defined orthonormal frame field on $(M,g)$.
 
 If $M$ is compact,
 we consider critical maps of the energy functional
 \begin{align}
 E(\p)=\int_M e(\p) v_g,
 \end{align}
where $e(\p)=\frac{1}{2}\|d\p\|^2=\sum^m_{i=1}\frac{1}{2}\langle d\p(e_i),d\p(e_i)\rangle$
 which is called the {\em energy density} of $\p$, and the inner product 
 $\langle \cdot ,\cdot \rangle$ is a Riemannian metric $h$. 
 The {\em tension \ field} $\tau(\p)$ of $\p$ is defined by
 \begin{align}
 \tau(\p)=\sum^{m}_{i=1}(\widetilde \nabla d\p)(e_i,e_i)=\sum^m_{i=1}(\widetilde \nabla _{e_i}d\p)(e_i).
 \end{align}
 Then, $\p$ is a {\em harmonic map} if $\tau(\p)=0$.
 
 The curvature tensor field $ R^N(\cdot, \cdot)$ of the Riemannian metric on the bundle 
$TN$ is defined as follows :
 \begin{align}
R^N(X,Y)
=\nabla^N_X \nabla^N_Y - \nabla^N _Y \nabla^N_X-\nabla^N_{[X,Y]},
\ \ \ \ (X,Y\in \Gamma (TN)).
\end{align}
 

$\olapla
=\onabla^* \onabla
=-\sum^m_{k=1}(\onabla_{e_k}\onabla_{e_k}
-\onabla_{\nabla_{e_k}e_k}),$ 
is the {\em rough Laplacian}.



And G.Y.Jiang \cite{jg1} showed that $\phi:(M,g)\rightarrow (N,h)$ is a biharmonic (2-harmonic) if and only if 
$$\olapla \tau(\p) -R^N(\tau(\p),d\p(e_i))d\p(e_i)=0.$$



\vspace{30pt}

\section{$k$-harmonic maps}\label{k-harmonic}

J. Eells and L. Lemaire \cite{jell1} proposed the notation of $k$-harmonic maps. 
The Euler-Lagrange equation for the $k$-harmonic maps was shown by
Wang Shaobo \cite{ws1}.
In this section, we recall $k$-harmonic maps (cf \cite{sm1},\cite{ws1}).

We consider a smooth variation $\{\p_{t}\}_{t\in I_{\epsilon}} (I_{\epsilon}=(-\epsilon, \epsilon))$ of $\p$ with parameter $t$,
 i.e. we consider the smooth map $F$  given by
$$F : I_{\epsilon}\times M\rightarrow N, F (t,p)=\p_{t}(p),$$
where $F (0,p)=\p_{0}(p)=\p(p), $ for all $p\in M$.

The corresponding variational vector field $V$ is given by
\begin{align*}
V(p)=\left.\frac{d}{dt}\right|_{t=0}\p_{t}(p)\in T_{\p(p)}N,
\end{align*}
$V$ are section of $\p^{-1}TN$, i.e., $V\in \Gamma (\p^{-1}TN)$.
 

\vspace{20pt}

\begin{defn}[\cite{jell1}]
For $k=1,2,\dotsm$ the {\em $k$-energy functional}
 is defined by 
\begin{align*}
E_k(\phi)=\frac{1}{2}\int_M\|(d+d^* )^k\phi\|^2v_g,\ \ \phi\in C^{\infty}(M,N).
\end{align*}
Then, $\phi$ is {\em $k$-harmonic} if it is a critical point of $E_k,$ i.e., for all smooth variation $\{\phi_t\}$ of $\phi$ with $\phi_0=\phi$,
\begin{align*}
\left.\frac{d}{dt}\right|_{t=0}E_k(\phi_t)=0.
\end{align*}
We say for a $k$-harmonic map to be {\em proper} if it is not harmonic. 
\end{defn}

\begin{thm}[\cite{ws1}]\label{2s-harmonic}
Let $k=2s\ (s=1,2,\cdots)$,
\begin{align*}
\left.\frac{d}{dt}\right|_{t=0}E_{2s}(\p_t)=-\int_M\langle \tau_{2s}(\p), V\rangle,
\end{align*}
where, 
\begin{align}\label{2s-harmonic eq}
\tau_{2s}(\p)
=&\olapla^{2s-1}\tau(\p)
-R^N(\olapla^{2s-2}\tau(\p),d\p(e_j))d\p(e_j)\notag\\
&-\sum^{s-1}_{l=1}
\{
R^N(\onabla_{e_j}\olapla^{s+l-2}\tau(\p),\olapla^{s-l-1}\tau(\p) )d\p(e_j) \\
&\hspace{32pt}-R^N(\olapla^{s+l-2}\tau(\p),\onabla_{e_j}\olapla^{s-l-1}\tau(\p) )d\p(e_j)
\}\notag,
\end{align}
where, $\olapla^{-1}=0$.
\end{thm}

\vspace{20pt}

\begin{thm}[\cite{ws1}]\label{2s+1-harmonic}
Let $k=2s+1\ \ \ (s=0,1,2,\cdots),$
\begin{align*}
\left.\frac{d}{dt}\right|_{t=0}E_{2s+1}(\p_t)=-\int_M\langle \tau_{2s+1}(\p), V\rangle,
\end{align*}
where, 
\begin{align}\label{2s+1-harmonic eq}
\tau_{2s+1}(\p)
=&\olapla^{2s}\tau(\p)
-R^N(\olapla^{2s-1}\tau(\p),d\p(e_j))d\p(e_j)\notag\\
&-\sum^{s-1}_{l=1}
\{
R^N(\onabla_{e_j}\olapla^{s+l-1}\tau(\p),\olapla^{s-l-1}\tau(\p) )d\p(e_j) \\
&\hspace{32pt}-R^N(\olapla^{s+l-1}\tau(\p),\onabla_{e_j}\olapla^{s-l-1}\tau(\p) )d\p(e_j)
\}\notag\\
&-R^N(\onabla_{e_i}\olapla^{s-1}\tau(\p),\olapla^{s-1}\tau(\p))d\p(e_i),\notag
\end{align}
where, $\olapla^{-1}=0.$
\end{thm}

\vspace{20pt}

\section{$k$-harmonic isometric immersion}\label{immersion}

In this section we consider $k$-harmonic isometric immersion $\p: (M,g)\rightarrow (N,h)$ from an $m$-dimensional compact Riemannian manifold into an $n$-dimensional Riemannian manifold. And in this section, we denote 
the second fundamental form $B(e_i,e_j)$ by $B_{ij}$, where $\{e_i\}_{i=1}^m$ is locally defined  orthonormal frame field on $M$.
 
G. Y. Jiang \cite{jg1} showed the following.

\begin{thm}[\cite{jg1}]\label{bi sphere}
Let $\p:M\rightarrow S^{m+1}$ be an isometric immersion having parallel mean curvature vector field with non-zero mean curvature.
 Then, the necessary and sufficient condition for $\p$ to be biharmonic is $||B||^2=m={\rm dim}M$.
\end{thm}

We consider the case of $k$-harmonic isometric immersion. First we show the next lemmas.

\begin{lem}\label{ljg1}
Let $\p:(M,g)\rightarrow (N,h)$ be an isometric immersion of which the second fundamental form $B$ is parallel,
 i.e., $\nabla^{\perp}B=0,$ where $\nabla^{\perp}$ is the induced connection of the normal bundle $T^{\perp}M$ by $\p$.
Then, 
\begin{align*}
\olapla B_{st}
=h(\olapla B_{st},d\p(e_i))d\p(e_i)
-h(\onabla_{e_i}B_{st},d\p(e_j))B_{ij}.
\end{align*}
\end{lem}

\begin{proof}
Let us recall the definition of $\nabla ^{\perp }$: For any section $\xi \in \Gamma (T^{\perp}M)$,
we decompose ${\onabla}_X{\xi }$ according to $TN{\mid }_M=TM{\oplus }T^{\perp }M$ as follows.

$$\onabla_X{\xi }=\nabla _{d\p (X)}^N{\xi} =\nabla _{d\p (X)}^T{\xi }+\nabla _{d\p (X)}^{\perp }\xi .$$

By the assumption, 
$\nabla _{d\p (X)}^{\perp }B_{st}=0$
for all $X\in {\Gamma}(TM)$, we have 

$$\onabla _{X}B_{st} =\nabla _{d\p (X)}^TB_{st} \in \Gamma (\p _{*}TM).$$

Thus, for all $i=1,\cdots ,m,$

$$\onabla _{e_i}B_{st} =h(\onabla _{e_i}B_{st} ,d\p ({e_j}))d\p ({e_j})$$
because $\{d\p (e_j)_x\}_{j=1}^m$ is an orthonormal basis with respect to $h$,
of $\p _*T_xM (x\in M)$. Now let us calculate
 
$$\onabla ^*\onabla B_{st} 
=-\{\onabla _{e_i}\onabla _{e_i}B_{st} -\onabla _{\nabla _{e_i}e_i}B_{st} \}.$$ 

Indeed, we have
\begin{align*}
\onabla _{e_i}\onabla _{e_i}B_{st} 
=&
h(\onabla _{e_i}\onabla _{e_i}B_{st},d\p(e_j))d\p(e_j) +h(\onabla _{e_i}B_{st},\onabla _{e_i}d\p (e_j))d\p (e_j)\\
&+
h(\onabla_{e_i}B_{st} ,d\p (e_j))\onabla_{e_i}d\p (e_i),
\end{align*}
and

$$\onabla_{\nabla _{e_i}e_i}B_{st} =h(\onabla_{\nabla _{e_i}e_i}B_{st},d\p (e_j))d\p (e_j),$$
so that we have

\begin{align}\label{21}
\onabla^*\onabla B_{st}
=&
h(\onabla^*\onabla B_{st},d\p (e_j))d\p (e_j)\\
&\hspace{10pt}-
\{h(\onabla_{e_i}B_{st} ,\onabla_{e_i}d\p (e_j))d\p (e_j)\notag\\
&\hspace{25pt}+h(\onabla_{e_i}B_{st} ,d\p (e_j))\onabla_{e_i}d\p (e_i)\}\notag.
\end{align} 

Denoting $\nabla _{e_i}e_j=\sum_{k=1}^m\Gamma _{ij}^ke_k,$ we have $\Gamma _{ij}^k+\Gamma _{ik}^j=0.$
Since $(\wnabla_{e_i}d\p )(e_j)=\onabla_{e_i}d\p (e_j)-d\p (\nabla _{e_i}e_j)$ is a local section of $T^\perp M$,
we have for the second term of the RHS of $(\ref{21})$, for each fixed $i=1,\dots,m,$

\begin{align}\label{22}
h(\onabla_{e_i}B_{st} ,\onabla_{e_i}d\p (e_j))d\p (e_j)
=&
h(\onabla_{e_i}B_{st} ,(\wnabla_{e_i}d\p )(e_j)+d\p (\nabla _{e_i}e_j))d\p (e_j)\\
=&
h(\onabla_{e_i}B_{st} ,d\p (\nabla _{e_i}e_j))d\p (e_j)\notag\\
=&
h(\onabla_{e_i}B_{st} ,d\p (e_k))d\p (\Gamma _{ij}^ke_j)\notag\\
=&
-h(\onabla_{e_i}B_{st} ,d\p (e_k))d\p (\Gamma _{ik}^je_j)\notag\\
=&
-h(\onabla_{e_i}B_{st} ,d\p (e_k))d\p (\nabla _{e_i}e_k)\notag.
\end{align}
Substituting $(\ref{22})$ into $(\ref{21})$, we obtain this lemma.
\end{proof}

\begin{lem}\label{ljg2}
Under the same assumption as in Lemma $\ref{ljg1}$, we have
\begin{align*}
\olapla B_{st}
=-h(B_{st},R^N(d\p(e_j),d\p(e_k))d\p(e_k))d\p(e_j)
+h(B_{st},B_{ij})B_{ij}.
\end{align*}
\end{lem}

\begin{proof}
Since $h(B_{st} ,d\p (e_j))=0,$ differentiating it by $e_i$, we have
\begin{align*}
h(\onabla_{e_i}B_{st} ,d\p (e_j))
=&-h(B_{st} ,\onabla_{e_i}d\p (e_j))\\
=&-h(B_{st} ,(\wnabla_{e_i}d\p )(e_j)).
\end{align*}

And we have 
\begin{align*}
h(\onabla_{e_i}\onabla_{e_i}B_{st},d\p(e_j))+h(\onabla_{e_i}B_{st},\onabla_{e_i}d\p(e_j))\\
=-h(\onabla_{e_i}B_{st},B_{ij})-h(B_{st},\onabla_{e_i}B_{ij})
\end{align*}

So we have

\begin{align}\label{4-2-1}
0=&h(\onabla_{e_i}\onabla_{e_i}B_{st},d\p(e_j))+h(\onabla_{e_i}B_{st},\onabla_{e_i}d\p(e_j))\\
&+h(\onabla_{e_i}B_{st},B_{ij})+h(B_{st},\onabla_{e_i}B_{ij})\notag\\
=&h(\onabla_{e_i}\onabla_{e_i}B_{st},d\p(e_j))+h(\onabla_{e_i}B_{st},d\p(\nabla_{e_i}e_j))\notag
+h(B_{st},\onabla_{e_i}(\wnabla_{e_i}d\p)(e_j))\notag\\
=&h(\onabla_{e_i}\onabla_{e_i}B_{st},d\p(e_j))+h(\onabla_{e_i}B_{st},d\p(\nabla_{e_i}e_j))\notag\\
&+h(B_{st},(\wnabla_{e_i}\wnabla_{e_i}d\p)(e_j)
+(\wnabla_{e_i}d\p)(\nabla_{e_i}e_j)).\notag
\end{align}

Here, using $h(B_{st},d\p(\nabla_{e_i}e_j))=0$, we get

\begin{align*}
0&=h(\onabla_{e_i}B_{st},d\p(\nabla_{e_i}e_j))+h(B_{st},\onabla_{e_i}d\p(\nabla_{e_i}e_j))\\
&=h(\onabla_{e_i}B_{st},d\p(\nabla_{e_i}e_j))+h(B_{st},(\wnabla_{e_i}d\p)(\nabla_{e_i}e_j)).\\
\end{align*}

Thus, $(\ref{4-2-1})$ coincides with

\begin{align}\label{4-2-2}
h(\onabla_{e_i}\onabla_{e_i}B_{st},d\p(e_j))
+h(B_{st},(\wnabla_{e_i}\wnabla_{e_i}d\p)(e_j)=0.
\end{align}

And using $h(B_{st},d\p(e_j))=0,$ we obtain 

\begin{align}\label{4-2-3}
0=&h(\onabla_{\nabla_{e_i}e_i}B_{st},d\p(e_j))+h(B_{st},\onabla_{\nabla_{e_i}e_i}d\p(e_j))\\
=&h(\onabla_{\nabla_{e_i}e_i}B_{st},d\p(e_j))+h(B_{st},(\wnabla_{\nabla_{e_i}e_i}d\p)(e_j))\notag.
\end{align}

By $(\ref{4-2-2}),$ and $(\ref{4-2-3})$, we have
\begin{align}\label{4-3-1}
0=&-h(\onabla_{e_i}\onabla_{e_i}B_{st},d\p(e_j))
-h(B_{st},(\wnabla_{e_i}\wnabla_{e_i}d\p)(e_j)\\
&+h(\onabla_{\nabla_{e_i}e_i}B_{st},d\p(e_j))+h(B_{st},(\wnabla_{\nabla_{e_i}e_i}d\p)(e_j))\notag\\
=&h(\olapla B_{st},d\p(e_j))+h(B_{st},(\wnabla^*\wnabla d\p)(e_j)),\notag
\end{align}

where, $\wnabla^*\wnabla=-(\wnabla_{e_i}\wnabla_{e_i}-\wnabla_{\nabla_{e_i}e_i})$.

By Wizenb\"{o}ck formula 

$$\lapla d\p=\wnabla^*\wnabla d\p+Sd\p,$$

where, 
$$Sd\p(X)=-(\widetilde R(X,e_k)d\p)(e_k),$$

\begin{align*}
\widetilde R(X,Y)d\p(Z)
=&R^{\p^{-1}TN}(X,Y)d\p(Z)-d\p(R^M(X,Y)Z)\\
=&R^N(d\p(X),d\p(Y))d\p(Z)-d\p(R^M(X,Y)Z),\ \ X,Y,Z\in \Gamma (TM),
\end{align*}

and 
$$\lapla d\p(e_j)=(dd^*+d^*d)d\p(e_j)=dd^*d\p(e_j)=-d\tau(\p)(e_j)=-\onabla_{e_j}\tau(\p),$$

we obtain 

\begin{align}\label{4-3-2}
(\wnabla^*\wnabla d\p)(e_j)
=&\lapla d\p(e_j)-Sd\p(e_j)\\
=&-\onabla_{e_j}\tau(\p)+R^N(d\p(e_j),d\p(e_k))d\p(e_k)-d\p(R^M(e_j,e_k)e_k).\notag
\end{align}

Substituting $(\ref{4-3-2})$ into $(\ref{4-3-1})$, we get

\begin{align*}
h(\olapla B_{st},d\p(e_j))
=&-h(B_{st},-\onabla_{e_j}\tau(\p)+R^N(d\p(e_j),d\p(e_k))d\p(e_k)\\
&\hspace{40pt}-d\p(R^M(e_j,e_k)e_k))\\
=&-h(B_{st}, R^N(d\p(e_j),d\p(e_k))d\p(e_k)).
\end{align*}
Therefore, we obtain this lemma.
\end{proof}

\begin{lem}\label{olapla^k B}
Let $\p:(M,g)\rightarrow (N,h)$ be an isometric immersion into a Riemannian manifold with constant sectional curvature $K$, 
and second fundamental form of $\p$ is parallel, Then,
\begin{equation}
\olapla^k B_{ij}=h(B_{ij},B_{i_{1}j_{1}})\prod^{k-1}_{l=1}h(B_{i_{l}j_{l}},B_{i_{(l+1) }j_{(l+1) }})B_{i_{k}j_{k}}.
\end{equation} 
\end{lem}

\begin{proof}
Since the curvature tensor $R^N$  of $(N,h)$ is given by 
$$R^N(U,V)W=K\{h(V,W)U-h(W,U)V\},\ \ \ U,V,W\in \Gamma(TN),$$
$R^N(d\p(e_j),d\p(e_k))d\p(e_k)$ is tangent to $\p_*TM$. By Lemma \ref{ljg2}, we have
\begin{align}
\olapla B_{st}=h(B_{st},B_{ij})B_{ij}
\end{align}


For all $X\in \Gamma(TM)$, we have

\begin{align*}
\onabla_X\{h(B_{st},B_{ij})B_{ij}\}
=&(\onabla_X h(B_{st},B_{ij}))B_{ij}
+h(B_{st},B_{ij})\onabla_X B_{ij}\\
=&h(B_{st},B_{ij})\onabla_X B_{ij}.
\end{align*}

Similarly we have

$$\onabla_X\onabla_X\{h(B_{st},B_{ij})B_{ij}\}=h(B_{st},B_{ij})\onabla_X\onabla_X B_{ij}.$$

Thus, we obtain

$$\olapla^2B_{st}=h(B_{st},B_{ij})\olapla B_{ij}
=h(B_{st},B_{ij})h(B_{ij},B_{kl})B_{kl}.$$

Repeating these steps, we get

\begin{equation}
\olapla^k B_{ij}=h(B_{ij},B_{i_{1}j_{1}})\prod^{k-1}_{l=1}h(B_{i_{l}j_{l}},B_{i_{(l+1) }j_{(l+1) }})B_{i_{k}j_{k}}.
\end{equation} 

Therefore, we obtain the lemma.

\end{proof}

Using these lemmas, we show the next proposition.

\begin{prop}\label{to sphere}
Let $\p:M\rightarrow S^{m+1}$ be an isometric immersion having parallel second fundamental form.
 Then, the necessary and sufficient condition for $\p$ to be $k$-harmonic is 
$$||B||^4-m||B||^2-(k-2)||\tau(\p)||^2=0.$$
\end{prop}

\begin{proof}
If we denote by $\xi$ the unit normal vector field to $\p(M)$, the second fundamental form $B$ 
 is of the form $B(e_i,e_j)=(\widetilde \nabla _{e_i}d\p)(e_j)=h_{ij}\xi$. Then, we have 
 $\tau(\p)=B(e_i,e_i)=h_{ii}\xi$ and $||B||^2=h_{ij}h_{ij}.$ Using Lemma $\ref{olapla^k B}$, we get 
 $\olapla^k \tau(\p)=||B||^{2k}\tau(\p)$. So we have
 
\begin{align*}
0=\tau_{2s}(\p)
=&\olapla^{2s-1}\tau(\p)-\{m\olapla^{2s-2}\tau(\p)\}\\
&-\sum^{s-1}_{l=1}\{(-\langle d\p(e_i),\onabla_{e_i}\olapla^{s+l-2}\tau(\p)\rangle \olapla^{s-l-1}\tau(\p))\\
&\hspace{30pt}-(\langle \onabla_{e_i}\olapla^{s-l-1}\tau(\p),d\p(e_i) \rangle \olapla^{s+l-2}\tau(\p) )\}\\
=&\olapla^{2s-1}\tau(\p)-m\olapla^{2s-2}\tau(\p)\\
&-\sum^{s-1}_{l=1}\{\langle \tau(\p),\olapla^{s+l-2}\tau(\p)\rangle \olapla^{s-l-1}\tau(\p)
+\langle \olapla^{s-l-1}\tau(\p),\tau(\p) \rangle \olapla^{s+l-2}\tau(\p) \}\\
=&||B||^{2(2s-1)}\tau(\p)-m||B||^{2(2s-2)}\tau(\p)\\
&-\sum^{s-1}_{l=1} \{ ||B||^{2(2s-3)}|| \tau(\p)||^2 \tau(\p)+||B||^{2(2s-3)}||\tau(\p)||^2\tau(\p)\}\\
=&||B||^{4s-2}\tau(\p)-m||B||^{4s-4}\tau(\p)-2(s-1)||B||^{4s-6}||\tau(\p)||^2\tau(\p)\\
=&||B||^{4s-6}\tau(\p)\{||B||^4-m||B||^2-2(s-1)||\tau(\p)||^2\}.
\end{align*}

So $\p$ is $2s$-harmonic $(s=1,2,\cdots)$ if and only if 

$$||B||^4-m||B||^2-2(s-1)||\tau(\p)||^2=0.$$

Similarly,
$\p$ is  $2s+1$-harmonic $(s=1,2,\cdots)$ if and only if

$$||B||^4-m||B||^2-(2s-1)||\tau(\p)||^2=0.$$
\end{proof}

\begin{cor}
Let $\p:M\rightarrow S^{m+1}$ be a $k$-harmonic isometric immersion $(k=3,4, \cdots)$ having parallel second fundamental form
 $B\neq 0$. If $\p$ is biharmonic, then $\p$ is harmonic.
\end{cor}

\begin{proof}
By Proposition \ref{to sphere},
$$0=m^2-m\cdot m-(k-2)||\tau(\p)||^2=-(k-2)||\tau(\p)||^2.$$
So we get the corollary.
\end{proof}

We define (cf. \cite{ws1},\cite{jg1})
$$M^m_{p}(\lambda)=S^p\left(\frac{1}{\sqrt{1+\lambda^2}}\right)\times S^{m-p}\left(\frac{\lambda}{\sqrt{1+\lambda^2}}\right)
\ \ \ \ \lambda>0, 0\leq p\leq m.$$

Indeed, $\p:M^m_p(\lambda)\rightarrow S^{m+1}$ has the parallel second fundamental form, and $\p$ satisfies that

$$||\tau(\p)||=|p\lambda-(m-p)\frac{1}{\lambda}|,$$
and 
$$||B||^2=p\lambda^2+(m-p)\frac{1}{\lambda^2}.$$

Thus, when $\lambda^2=\frac{m-p}{p}$, $\p$ is harmonic.

Using Proposition \ref{to sphere}, $\p$ is $k$-harmonic if and only if 


$$(\lambda^2-\frac{m-p}{p})(\lambda^6-(k-1)\lambda^4+(k-1)\frac{(m-p)}{p}\lambda^2-\frac{m-p}{p})=0.$$

Thus $\p$ is proper $k$-harmonic if and only if

$$\lambda^6-(k-1)\lambda^4+(k-1)\frac{(m-p)}{p}\lambda^2-\frac{m-p}{p}=0,\ \ \lambda^2\neq \frac{m-p}{p}.$$

So we obtain two theorems.
\begin{thm}
Isometric immersion $\p:S^m(\frac{1}{\sqrt{k}})\rightarrow S^{m+1}$ is a  special proper  $k$-harmonic map .
\end{thm}

\begin{thm}
Isometric immersion $\p:M^m_{p}(\lambda)\rightarrow S^{m+1}$ is a proper $k$-harmonic map,
where $\lambda$ satisfies that
$$\lambda^6-(k-1)\lambda^4+(k-1)\frac{(m-p)}{p}\lambda^2-\frac{m-p}{p}=0,\ \ \lambda^2\neq \frac{m-p}{p}.$$
\end{thm}

Especially, we have the next two results

\begin{cor}

When, $m=2p$,

1)When, $k=2,3,4$.

There are no proper $k$-harmonic map.

2)When, $k=5,6,\cdots$.

 Isometric immersion $\p:M^m_{p}(\lambda)\rightarrow S^{m+1}$ is a proper $k$-harmonic map 
 if and only if 
$$\lambda=\sqrt{\frac{k-2\pm\sqrt{k(k-4)}}{2}}.$$
\end{cor}

\begin{cor}
When, $m\neq2p$, isometric immersion $\p:M^m_{p}(\lambda)\rightarrow S^{m+1}$ is proper $k$-harmonic,
where, 
$$\lambda=\sqrt{A^{\frac{1}{3}}+\frac{k^2-3a(k-1)-2k+1}{9A^{\frac{1}{3}}}+\frac{k-1}{3}},$$
where,
\begin{align*}
A=&\frac{\sqrt{-a(ak^4-4(a^2+a+1)k^3+12(a+1)^2k^2-4(3a^2+10a+3)k+4(a-1)^2)}}{23^\frac{3}{2}}\\
&-\frac{-2k^3+9a(k^2-2k-2)+2(3k^2-3k+1)}{54},
\end{align*}
\begin{align*}
a=\frac{m-p}{p}.
\end{align*}
\end{cor}

\section{$k$-harmonic submersion}\label{submersion}

In this section we generalize Oniciuc's results \cite{co1}.
Let 

\begin{align*}
&S^n(a)=S^n(a)\times \{b\}\\
&=\{p=(x^1,\cdots , x^{n+1},b);\ (x^1)^2+\cdots +(x^{n+1})^2=a^2,a\in (0,1),\ a^2+b^2=1\}
\end{align*}
be a parallel hypersurface of $S^{n+1}$, and the canonical metric $\langle \cdot, \cdot \rangle$ on $S^{n+1}$.

And let 
\begin{align*}
\Gamma(TS^n(a))=\{X=(X^1,\cdots , X^{n+1},0)\ x^1 X^1+\cdots +x^{n+1}X^{n+1}=0\},
\end{align*}

be the set of all sections of the tangent bundle of $S^n(a)$.

Then,
$$\eta=\frac{1}{c}(x^1,\cdots , x^{n+1},-\frac{a^2}{b}),$$
is a unit section in the tangent bundle of $S^n(a)$ in $S^{n+1}$, where $c>0$ and $c^2=a^2+\frac{a^4}{b^2}$ .

Thus $\eta$ satisfies that

$$\langle \eta , p \rangle =0,\ \ \ \ 
\langle \eta , X \rangle =0\ \ \ 
|\eta|=1.$$

By a direct computation we obtain 
\begin{align}\label{2.1}
A=-\frac{1}{c}I,\ \ B(X,Y)=-\frac{1}{c}\langle X,Y\rangle \eta,\ \ \nabla^{\perp}\eta=0,
\end{align}
where $A$ is the shape operator, $B$ the second fundamental form of $S^n(a)$ and $\nabla^{\perp}$
 is the normal connection in the normal bundle of $S^n(a)$ in $S^{n+1}$.

Now, we consider a Riemannian submersion $\varphi : (M,g)\rightarrow S^n(a)$, the canonical inclusion 
$i:S^n(a)\rightarrow S^{n+1}$, and 
$\p=i\circ \varphi :(M,g)\rightarrow S^{n+1}$. 
 The rank of $\p$ is constant, equal to $n$. 
The next theorem was proved by C. Oniciuc \cite{co1}.
\begin{thm}[\cite{co1}]
Assume that $\varphi:(M,g)\rightarrow S^n(a)$ is a harmonic Riemannian submersion. Then, 
$\p:(M,g)\rightarrow S^{n+1}$ is not harmonic, and it is biharmonic if and only if $a=\frac{1}{\sqrt{2}}$ and $b=\pm\frac{1}{\sqrt{2}}$. 
\end{thm}

We generalize this theorem.

\begin{prop}\label{2s Sa}
Assume that $\varphi:(M,g)\rightarrow S^n(a)$ is a harmonic Riemannian submersion. Then, 
 $\p:(M,g)\rightarrow S^{n+1}$ is proper $2s$-harmonic if and only if
$a=\frac{1}{\sqrt{2s}}$ and $b=\pm\sqrt{\frac{2s-1}{2s}}$. 
\end{prop}

\begin{proof}
Let $p\in M$, we have $T_pM=T^V_pM\oplus T^H_pM$, where 
$T^V_pM={\rm ker}d\vp _p$ and $T^H_pM$ is the ortogonal complement of $T^V_pM$ in $T_pM$
 with respect to the metric $g$. Let $W$ be an open subset of $S^n(a)$ such that $\vp (p)\in W$ 
 and let $\{Y_\alpha\}^n_{\alpha =1}$ be an orthonormal frame field of $W$. Set $U=\vp^{-1}(W)$,
 $\{X_{\alpha}\}=\{Y_{\alpha}^H\}$, and consider an orthonormal frame field $\{X_s\}^m_{s=n+1}$ on 
 $T^VU$. The tension field of $\vp$ is given by 
\begin{align}
\tau(\vp)_p=-\sum^m_{s=n+1}d\vp _p(\nabla_{X_s}X_s)
\end{align}  
(cf \cite{co1}).
Computing the tension field of $\p$ we obtain 
$$\tau(\p)=di(\tau(\vp))+trace\nabla di (d\vp\cdot , d\vp \cdot)=\sum^n_{\alpha =1}B(Y_{\alpha},Y_{\alpha})=-\frac{n}{c}\eta,$$
i.e., $\p$ is not harmonic.

To simplify the notation, we denote the Levi-Civita connection $\nabla^{S^n(a)}$ of $S^n(a)$ by $\nabla^N.$
Computing $\olapla \eta$ we get
\begin{align*}
\olapla \eta
=&-\sum^m_{i=1}\{\onabla_{X_i}\onabla_{X_i}\eta-\onabla_{\nabla_{X_i}X_i}\eta\}\\
=&-\sum^n_{\alpha=1}\{\onabla_{X_{\alpha}}\onabla_{X_{\alpha}}\eta-\onabla_{\nabla_{X_{\alpha}}X_{\alpha}}\eta\}\\
\hspace{10pt}&-\sum^m_{s=n+1}\{\onabla_{X_s}\onabla_{X_s}\eta-\onabla_{\nabla_{X_s}X_s}\eta\}.
\end{align*}
Here we obtain
$$\onabla_{X_{\alpha}}\eta=\nabla^{S^{n+1}}_{Y_{\alpha}}\eta=\frac{1}{c}Y_{\alpha},$$
and using $(\ref{2.1})$ we obtain
\begin{align} 
\onabla_{X_{\alpha}}\onabla_{X_{\alpha}}\eta
=&\frac{1}{c}\nabla^{S^{n+1}}_{Y_{\alpha}}Y_{\alpha}
=\frac{1}{c}(\nabla^N_{Y_{\alpha}}Y_{\alpha}-\frac{1}{c}\eta)\\
=&\frac{1}{c}\nabla^N_{Y_{\alpha}}Y_{\alpha}-\frac{1}{c^2}\eta\notag.
\end{align}
Further, we have
\begin{align}
\onabla_{\nabla_{X_{\alpha}}X_{\alpha}}\eta
=&\nabla^{S^{n+1}}_{\nabla^N_{Y_{\alpha}}Y_{\alpha}}\eta=\frac{1}{c}\nabla^N_{Y_{\alpha}}Y_{\alpha},
\end{align}
\begin{align}
\onabla_{X_s}\onabla_{X_s}\eta=0,
\end{align}
\begin{align}
\onabla_{\nabla_{X_s}X_s}\eta
=&\nabla^{S^{n+1}}_{d\vp(\nabla_{X_s}X_s)}\eta=\frac{1}{c}d\vp(\nabla_{X_s}X_s).
\end{align}

So we obtain
\begin{align}
\olapla \eta=\frac{n}{c^2}\eta.
\end{align}
Repeating these steps, we have
\begin{align}
\olapla^l \eta=(\frac{n}{c^2})^l\eta,\ \ \  (l=1,2,\cdots ),
\end{align}
and 
\begin{align}\label{ola}
\olapla^l \tau(\p)=-\frac{n}{c}(\frac{n}{c^2})^l\eta,\ \ \  (l=1,2,\cdots).
\end{align}
A direct computation shows
\begin{align}\label{R1}
\sum^m_{i=1}R^{S^{n+1}}(\olapla^l\tau(\p),d\p(X_i))d\p(X_i)=n\olapla^l\tau(\p)=-\frac{n^2}{c}(\frac{n}{c^2})^l\eta,
\end{align}
\begin{align}\label{R2}
\sum^m_{i=1}R^{S^{n+1}}(\onabla_{X_i}\olapla^l\tau(\p),\olapla^r\tau(\p))d\p(X_i)=-\frac{n^2}{c}(\frac{n}{c^2})^{l+r+1}\eta,
\end{align}
and
\begin{align}\label{R3}
\sum^m_{i=1}R^{n+1}(\olapla^l\tau(\p),\onabla_{X_i}\olapla^r\tau(\p))d\p(X_i)=\frac{n^2}{c}(\frac{n}{c^2})^{l+r+1}\eta.
\end{align}
Substituting $(\ref{ola}),\ (\ref{R1}),\ (\ref{R2})$ and $(\ref{R3})$ into $(\ref{2s-harmonic eq})$, we obtain
\begin{align}
\tau_{2s}(\p)
=&-\frac{n}{c}(\frac{n}{c^2})^{2s-1}\eta+\frac{n^2}{c}(\frac{n}{c^2})^{2s-2}\eta\\
&-\sum^{s-1}_{l=1}\{-\frac{n^2}{c}(\frac{n}{c^2})^{2s-2}\eta-\frac{n^2}{c}(\frac{n}{c^2})^{2s-2}\eta\}\notag\\
=&-\frac{n}{c}(\frac{n}{c^2})^{2s-1}\{1-(2s-1)c^2\}\eta.\notag
\end{align}
So $\p$ is proper $2s$-harmonic if and only if $c=\frac{1}{\sqrt{2s-1}}$ i.e., $a=\frac{1}{\sqrt{2s}}$ and $b=\pm\sqrt{\frac{2s-1}{2s}}$.  
\end{proof}

Similarly we have

\begin{prop}\label{2s+1 Sa}
Assume that $\varphi:(M,g)\rightarrow S^n(a)$ is a harmonic Riemannian submersion. Then, 
 $\p:(M,g)\rightarrow S^{n+1}$ is proper $2s+1$-harmonic if and only if
$a=\frac{1}{\sqrt{2s+1}}$ and $b=\pm\sqrt{\frac{2s}{2s+1}}$. 
\end{prop}

\begin{proof}
The proof is similar to that of Proposition \ref{2s Sa}.
\end{proof}

By Proposition \ref{2s Sa}, \ref{2s+1 Sa}, we obtain the next theorem.
\begin{thm}
Assume that $\varphi:(M,g)\rightarrow S^n(a)$ is a harmonic Riemannian submersion. Then, 
$\p:(M,g)\rightarrow S^{n+1}$ is proper $k$-harmonic if and only if
$a=\frac{1}{\sqrt{k}}$ and $b=\pm\sqrt{\frac{k-1}{k}}$. 
\end{thm}

\begin{proof}
By Proposition $\ref{2s Sa}$ and $\ref{2s+1 Sa}$, we obtain this theorem.
\end{proof}

Since the radial projection
$$S^n\rightarrow S^n(a),\ \ \ \ x\mapsto ax,$$
is homothetic, a harmonic Riemannian submersion $\vp:(M,g)\rightarrow S^n$
 becomes a harmonic Riemannian submersion $\vp:(M,a^2g)\rightarrow S^n(a),$
 and using the above theorem, we obtain a proper $k$-harmonic submersion $\p:M\rightarrow S^{n+1}.$ 
 Especially we obtain the next corollary.

\begin{cor}
The Hopf map 
$\vp:S^3(\sqrt{k})=\{(z^1,z^2)\in {\mathbb C}^2;\ (z^1)^2+(z^2)^2=k\}\rightarrow S^2(\frac{1}{\sqrt{k}})$,
that is 
$$\vp(z^1,z^2)=\frac{1}{k\sqrt{k}}(2z^1\bar{z^2},|z^1|^2-|z^2|^2)$$
induces a proper $k$-harmonic map.
\end{cor}

Let $n_1,n_2$ be two positive integers such that $n=n_1+n_2$ and let $r_1,r_2$ be two positive real numbers such that 
 $r_1^2+r_2^2=1$.
 Let $\vp_j:(M_j,g_j)\rightarrow S^{n_j}(r_j)$ $(j=1,2)$ be harmonic Riemannian submersions,
 and $\p=i\circ(\vp_1\times \vp_2)$, where $i:S^{n_1}(r_1)\times S^{n_2}(r_2)\rightarrow S^{n+1}$ is the canonical inclusion.
 
C. Oniciuc \cite{co1} showed the following.
 
\begin{thm}[\cite{co1}]
The map $\p$ is a proper biharmonic submersion if and only if $r_1=r_2=\frac{1}{\sqrt{2}}$ and $n_1\neq n_2$.
\end{thm}

We consider general case, and get the next results.

\begin{thm}
The map $\p$ is a proper $k$-harmonic submersion if and only if

$$(\frac{r^2_2}{r^2_1}n_1+\frac{r^2_1}{r^2_2}n_2)^2-n(\frac{r_2^2}{r_1^2}n_1+\frac{r_1^2}{r_2^2}n_2)
-(k-2)(\frac{r_1}{r_2}n_2-\frac{r_2}{r_1}n_1)^2=0,$$
and
$$\frac{r_1}{r_2}n_2-\frac{r_2}{r_1}n_1\neq0.$$
\end{thm}

\begin{proof}
Let 
$$\xi(p)=\left(\frac{r_2}{r_1}p_1,-\frac{r_1}{r_2}p_2\right),$$
where $p=(p_1,p_2)\in S^{n_1}(r_1)\times S^{n_2}(r_2).$
Then $\xi$ is a unit section in the normal bundle of $S^{n_1}(r_1)\times S^{n_2}(r_2)$ in $S^{n+1}$.
By a straightforward computation we obtain 
$$\tau(\p)=(\frac{r_1}{r_2}n_2-\frac{r_2}{r_1}n_1)\xi,$$
$$\tau_2(\p)=\frac{r^2_2-r^2_1}{r_1r_2}|\tau(\p)|^2\xi,$$
When, $k=3,4,\cdots,$
\begin{align}
\tau_k(\p)
=&(\frac{r^2_2}{r^2_1}n_1+\frac{r^2_1}{r^2_2}n_2)^{k-3}(\frac{r_1}{r_2}n_2-\frac{r_2}{r_1}n_1)\\
&\{(\frac{r^2_2}{r^2_1}n_1+\frac{r^2_1}{r^2_2}n_2)^2-n(\frac{r_2^2}{r_1^2}n_1+\frac{r_1^2}{r_2^2}n_2)
-(k-2)(\frac{r_1}{r_2}n_2-\frac{r_2}{r_1}n_1)^2\}\xi,\notag
\end{align}
Therefore, we have the theorem.
\end{proof}

\begin{cor}
i)\ If $n_2=n_1$, 

1)When, $k=2,3,4$

There are no proper $k$-harmonic submersion.

2)When, $k=5,6,\cdots $ 

The map $\p$ is a proper $k$-harmonic submersion if and only if
$$r_1=\sqrt\frac{1+\sqrt{\frac{k-4}{k}}}{2},r_2=\sqrt\frac{1-\sqrt{\frac{k-4}{k}}}{2},$$
or
$$r_1=\sqrt\frac{1-\sqrt{\frac{k-4}{k}}}{2},r_2=\sqrt\frac{1+\sqrt{\frac{k-4}{k}}}{2}.$$
\end{cor}
\vspace{5pt}
\begin{ex}

i)\ If $n_2=2n_1$ 

The map $\p$ is a proper $3$-harmonic submersion if and only if
$$r_1=\frac{\sqrt{3^{\frac{1}{6}}(81\sqrt{19}+197\sqrt{3})^{\frac{2}{3}}
+113^{\frac{1}{3}}(81\sqrt{19}+197\sqrt{3})^{\frac{1}{3}}
-14\sqrt{3}}}{3^{\frac{5}{3}}(81\sqrt{19}+197\sqrt{3})^{\frac{1}{6}}},$$
$$r_2=\frac{\sqrt{-3^{\frac{1}{6}}(81\sqrt{19}+197\sqrt{3})^{\frac{2}{3}}
+163^{\frac{1}{3}}(81\sqrt{19}+197\sqrt{3})^{\frac{1}{3}}
+14\sqrt{3}}}{3^{\frac{5}{3}}(81\sqrt{19}+197\sqrt{3})^{\frac{1}{6}}},$$

\begin{flushright}
$\cdots$ etc.
\end{flushright}

\vspace{5pt}

ii)\ If $n_2=cn_1,\ \ (c=2,3,\cdots)$

The map $\p$ is a proper $4$-harmonic submersion if and only if
$$r_1=\sqrt{\frac{(c-1)^{\frac{5}{3}}+(c+3)C^{\frac{1}{3}}+(c-1)^{\frac{1}{3}}C^{\frac{2}{3}}}{4C^{\frac{1}{3}}(c+1)}},$$
$$r_2=\sqrt{\frac{-(c-1)^{\frac{5}{3}}+(3c+1)C^{\frac{1}{3}}-(c-1)^{\frac{1}{3}}C^{\frac{2}{3}}}{4C^{\frac{1}{3}}(c+1)}},$$
where, $C=c^2+4\sqrt{c}(c+1)+6c+1$.
\end{ex}

\begin{flushleft}
\textbf{Acknowledgements}
\end{flushleft}
The author wishes to thank the {\em Research Fellowships of the Japan Society for the Promotion of Science for Young Scientists},
 which made possible his study and life of the author's family.
 
We would like to express our gratitude to Professor Hajime Urakawa who introduced and helped to accomplish this paper.



\end{document}